\documentclass[12pt,a4paper]{article}
\usepackage{amsmath}
\usepackage{amssymb}
\usepackage{amsthm}
\usepackage{amsfonts}
\usepackage{latexsym}
\usepackage{verbatim}

\theoremstyle{plain}
\newtheorem{thm}{Theorem}[section]

\newtheorem{lem}{Lemma}[section]
\newtheorem{prop}{Proposition}[section]

\theoremstyle{remark}

\theoremstyle{definition}

\newtheorem{rem}{Remark}[section]

\title{Equivariant asymptotics of Szeg\H{o} kernels under Hamiltonian $SU(2)\times S^1$-actions}

\author{Andrea Galasso \footnote{\noindent{\bf Address:} Room 407, Chee-Chun Leung Cosmology Hall, National Taiwan University; {\bf e-mail}: andrea.galasso@ncts.ntu.edu.tw}}
\date{}

\begin{document} 
\maketitle

\begin{abstract} Let $M$ be complex projective manifold and $A$ a positive line bundle on it. Assume that $G=SU(2)$ acts on $M$ in a Hamiltonian and holomorphic manner and that this action linearizes to $A$. Then, there is an associated unitary representation of $G$ on the Hardy space $H(X)$, where $X$ is the circle bundle inside the dual of $A$. The standard circle action on $H(X)$ commutes with the action of $G$ and thus one has a decomposition labeled by $(k{\nu},\,k)$, where $k\in\mathbb{Z}$ and ${ \nu }\in \hat{G}$. We consider the local and global asymptotic properties of the corresponding equivariant projector as $k$ goes to infinity. More generally, for a compact connected Lie group, we compute the asymptotics of the dimensions of the corresponding isotypes. \end{abstract}

\section{Introduction}

Let $(M,\,\omega)$ be a compact connected $d$-dimensional K\"ahler manifold, with quantum line bundle $(A,\,h)$ on it, $h$ is the Hermitian metric. The curvature form of the connection $\nabla$, compatible with both the holomorphic structure and the metric, is $-2\,\imath\,\omega$. Thus $M$ has a natural choice
of a volume form, given by $\mathrm{dV}_M:=(1/d!)\,\omega^{\wedge d}$. Let $A^\vee$ be the dual line bundle, and $X\subset A^\vee$ the unit circle bundle, with projection $\pi:X\rightarrow M$. Then $X$ is a contact and CR manifold by positivity of $A$; if $\alpha$ is the contact form, $X$ inherits the volume form
$\mathrm{dV}_X:=(2\pi)^{-1}\,\alpha\wedge \pi^*(\mathrm{dV}_M)$. 

Let $G$ be a compact connected Lie group $G$ and let $\mu :G\times M\rightarrow M$ be a holomorphic and Hamiltonian action. The moment map is $\Phi:M\rightarrow \mathfrak{g}^\vee$  (where $\mathfrak{g}$ is the Lie algebra of $G$). The Hamiltonian action $\mu$ naturally induces an infinitesimal contact action of $\mathfrak{g}$ on $X$ (see \cite{k}); explicitly, if $\xi\in \mathfrak{g}$ and $\xi_M$ is the corresponding Hamiltonian vector field on $M$, then its contact lift $\xi_X$ is
\[
\xi_X:= \xi_M^\sharp-\langle \Phi\circ \pi, \xi\rangle \,\partial_\theta\,,
\]
where $\xi_M^\sharp$ denotes the horizontal lift on $X$ of a vector field $\xi_M$ on $M$, and $\partial_\theta$ is the generator of the structure circle action on $X$.

Furthermore, suppose that the action $\mu$ lifts to an action 
$$\widetilde{\mu}:G\times X\rightarrow X$$
acting via contact and $CR$ automorphism on $X$. Under these assumptions, there is a naturally induced unitary representation of $G$ on the Hardy space
$H(X)\subset L^2(X)$; hence $H(X)$ can be equivariantly decomposed over the irreducible representations of $G$:
\[
H(X)=\bigoplus _{{\nu}\in \widehat{G}} H(X)_{{\nu}},
\]
where $\widehat{G}$ is the collection of all irreducible representations of $G$.
As is well-known, if $\Phi (m)\neq 0$ for every $m\in M$, then each isotypical component $H(X)_{{\nu}}$
is finite dimensional (see e.g. \S 2 of \cite{pao-IJM}).

Since the $G$-action on $X$ and the standard circle action commute, there is a global action of $S^1\times G$ on $X$ and thus an induced unitary representation
of $S^1\times G$ on $H(X)$. The irreducible representations of
$S^1\times G$ are labelled by $({\nu},j)$, where $j\in \mathbb{Z}$ and
${\nu}=(\nu,0)$, $\nu>0$; hence we have an equivariant unitary
decomposition
\begin{equation} \label{eq: decomposition}
H(X)=\bigoplus_{j=0}^{+\infty}\bigoplus_{{\nu}\in \hat{G}}H(X)_{({\nu},\,j)},
\quad \text{where}\quad H(X)_{({\nu},\,j)}:=H(X)_j\cap H(X)_{{\nu}}.
\end{equation}

Let $\Pi_{{\nu},j}:H(X)\rightarrow H(X)_{{\nu},j}$ 
denote the corresponding equivariant Szeg\H{o} projection; we denote with $\Pi_{{\nu},j}\in \mathcal{C}^\infty (X\times X)$ its distributional kernel. We are led to investigate the local asymptotics of $\Pi_{(k\,{\nu},k)}$ when $k$ drifts to infinity. 

We will consider the case $G=SU(2)$; the lifting of the action on $X$ always exists in this case. Let us introduce some notations. For $m\in M$, $\Phi(m)\in \mathfrak{g}$ is traceless skew-Hermitian $2\times 2$ matrix. We will identify $\mathfrak{g}\cong \mathfrak{g}^{\vee}$, $\mathcal{O}_{\nu}$ will denote the adjoint/co-adjoint orbit whose positive eigenvalue is $\nu$. Suppose that $\Phi(m)\neq 0$ for every $m\in M$; then $h_m\,T\in G/T$ will denote the unique coset such that 
\begin{equation}
\label{eqn:defn of hm}
\Phi(m) =\imath\, h_m
\begin{pmatrix}
\lambda (m)& 0 \\
0 & -\lambda (m)
\end{pmatrix}\, h_m^{-1};
\end{equation}
where, clearly, the assignments $\lambda:M\rightarrow (0,+\infty)$ and $m\in M\mapsto h_m\,T\in G/T$ are $\mathcal{C}^\infty$.

In order to expose some results we need to define some loci in $M$. Let us pose $M_{{\nu}}=\lambda^{-1}(\nu)$, $M_{\text{in}}$ the inverse image of the open set $(0,\nu)$ via $\lambda$ and finally $M_{\text{out}}=M\setminus \overline{M_{\text{in}}}$. In Section \S\ref{sec:geometric} we will clarify the geometry of these loci which are similar to those defined in \cite{gp}. In particular, suppose that $\nu$ is a regular value for $\lambda$, which is equivalent to impose $\Phi$ transversal to $\mathcal{O}_{\nu}$, we will prove that $M_{{\nu}}$ is a compact connected manifold and divides $M$ in two open connected components $M_{\text{in}}$ and $M_{\text{out}}$. 

Let us denote with $X_{{\nu}}$ the pull-back on $X$ of the equivariant locus $M_{{\nu}}$. Let us define the invariant subset of $X\times X$
	\[\mathcal{Z}_{\nu} := 
\{(x, y) \in X_{{\nu}} \times X_{{\nu}} \,:\, y\in (G\times S^1)\cdot x\}\,. \]
Using functorial properties of distributions  (in a similar way as in Section \S3 of \cite{gp}) one can prove that, uniformly on compact subsets of the locus $(X\times X)\setminus \mathcal{Z}_{\nu}$, one has $\Pi_{(k{\nu},k)}(x,y)=O\left(k^{-\infty}\right)$. In particular, for each $x\in X\setminus X_{{\nu}} $, the kernel $\Pi_{(k{\nu},k)}(x,x)$ has rapidly decreasing asymptotic.

Our first theorem concerns the rapidly decreasing asymptotics of the kernel $\Pi_{(k{\nu},k)}(x,y)$ when $(x,y)$ approaches to some loci in $M$ at a sufficiently fast pace. Let $\mathrm{dist}_X$ be the Riemannian distance function on $X$.

\begin{thm} \label{thm:rapid decrease}
Suppose that ${0}\notin \Phi(M)$ and $\Phi$ is transversal to the orbit $\mathcal{O}_{\nu}$. Let us fix $C,\,\epsilon>0$. Then, uniformly for $(x,y)\in X\times X$ satisfying
$$ 
\max\{\mathrm{dist}_X\left(x, (G\times S^1)\cdot y\right),\, \mathrm{dist}_X\left(x, \overline{ X_{\mathrm{in}}}\right)\}\ge C\, k^{\epsilon -1/2},
$$
we have 
$\Pi_{(k{\nu},k)}(x,y)=O\left(k^{-\infty}\right)$.
\end{thm} 

To build-up to our next theorems, we need to introduce some more terminology. Let us denote by 
\begin{equation}
\label{eqn:defn beta}
\beta:=
\begin{pmatrix}
\imath & 0\\
0 &-\imath
\end{pmatrix}\,,
\end{equation}
the infinitesimal generator of the standard torus $T$, and by
\begin{equation}
 \label{eqn:DG/T}
D_{G/T}
:=2\pi/V_3.
\end{equation}
where $V_3$ is the area of the unit sphere $S^3\subset \mathbb{R}^4$.

\begin{thm} Suppose that ${0}\notin \Phi(M)$,  $\Phi$ is transversal to the orbit $\mathcal{O}_{\nu}$ and fix $x\in X_{{\nu}}$ with $m:=\pi(x)$. Let us suppose that the stabilizer of $x$ is trivial. Then, we have
 \label{thm:diagonal}
 \begin{align*} 
 \Pi_{(k{\nu},\,k)}(x,\,x)\sim\sqrt{2}\,D_{G/T}\,\left( \frac{k}{\pi} \right)^{d-1/2}\cdot \frac{1}{\lVert \mathrm{Ad}_{h_m}\left(\beta\right)_M(m) \rVert} +O(k^{d-3/4}) \,. 
 \end{align*}
\end{thm}

The proof of Theorem \ref{thm:diagonal} is based on the study of an oscillatory integral over $G\times S^1$ whose phase function has critical points if
\[\langle \mathrm{Ad}_{g^{-1}}\Phi(m),\,\xi \rangle=0 \text{ for each } \xi \in \mathfrak{t}\,. \]
By invoking the Horn theorem this condition is equivalent to require that $\nu\leq \lambda(m)$; in the statement of the theorem we fix $x\in {X_{{\nu}}}$, id est $\lambda(m)=\nu$. Although in a different perspective, a similar situation appears in \cite{gp} where a lower bound for the dimension of the isotypes is given. In fact, in the proof of Theorem \ref{thm:diagonal}, we work in local Heisenberg coordinates centred at $x$ along orthogonal outer oriented directions to the locus $X_{{\nu}}$, see expansion \eqref{eq:finalPi}. This last equation together with Theorem \ref{thm:rapid decrease} allow one to give a lower bound regarding the asymptotics of the dimension of the isotypes.  In order to produce a more precise result concerning the asymptotics of the dimension, we use a different strategy. 

From now on, suppose that $G$ is a compact connected semi-simple Lie group of dimension $\mathrm{g}$. Let $\mathcal{O}_{{\nu}}^{-}$ be the coadjoint orbit in $\mathfrak{g}^{\vee}$ passing through the weight ${\nu}$ with the opposite of Kirillov-Kostant-Souriau form $\omega_{\mathcal{O}}$. Let $\pi_1\,:\,
M\times \mathcal{O}_{{\nu}}^-\rightarrow M$ be the projection onto $M$ and $\pi_2$ the projection onto $\mathcal{O}_{{\nu}}^-$. Then $\tilde{\omega}=\pi_1^*\,\omega-\pi_2^*\,\omega_{\mathcal{O}}$ is a symplectic form on $N:=M\times \mathcal{O}_{{\nu}}^-$. Furthermore $N$ is a Hamiltonian $G$-space and its moment mapping, $\tilde{\Phi}:\,M\times \mathcal{O}_{{\nu}}^{-} \rightarrow \mathfrak{g}^{\vee}$, is given by
\[\tilde{\Phi}(x,\xi)=\Phi(x)-\xi\,. \]
Let us denote with $N_0=(M\times \mathcal{O}_{{\nu}}^-)_0$ the zero locus of the moment map $\tilde{\Phi}$. Suppose that the action of $G$ on $N_0$ is free, then the \textit{Marsden-Weinstein reduction} of $M$ with respect to $\mathcal{O}_{{\nu}}$ is
\[M_{\mathrm{red}}=(M\times \mathcal{O}_{{\nu}}^-)_0/G\,. \]

Now, denote with $B$ the tensor product bundle on $N$. Thus, $B^{\otimes k}$, $\nabla_B^{(k)}$ and $(\cdot,\,\cdot)_B^{(k)}$ are $G$-invariant pre-quantum data on the manifold $(N,\,k\,\tilde{\omega})$. Let $H^0(N,B^{\otimes k})$ be the space of polarized sections of $B^{\otimes k}$ and $\mu_{k{\nu}}(V)$ be the multiplicity with which the irreducible representation
of $G$ corresponding to $\mathcal{O}_{k{\nu}}$ occurs in a vector space $V$. By the results in Section \S6 of \cite{gs-geom quant}, we have
\begin{align} \label{eq:gs2}
  \dim (H^0(N,B^{\otimes k})_0) 
 &= \mu_{k{\nu}}( H^0(M,A^{\otimes k})) =  \mu_{k{\nu}}(H(X)_{k})\,,  \end{align}
where, given a vector $G$-space $V$, $V_0$ denotes the space of $G$-fixed vectors. 

Let $\tilde{\Pi}$ be the Szeg\H{o} projector of $Y\subseteq B^{\vee}$, the circle bundle on $N$, and set $y,y'\in Y$. The asymptotics of the equivariant Szeg\H{o} kernels
\begin{equation} \label{eq:Pitilde}
\tilde{\Pi}_{0,k}(y,y')=\int_G \tilde{\Pi}_k(\tilde{\mu}_{g^{-1}}(y),\,y')\, \mathrm{dV}_G(g)
\end{equation}
have already been studied in \cite{pao-jsg0}. Let us recall some terminology. At any $n \in N$, let us denote by $\mathfrak{g}_N(n)\subseteq T_nN $ the tangent space to the orbit through $n$, and by $\tilde{J}_n : T_n N \rightarrow T_nN$ the complex structure. Furthermore, if $n\in N_0$ let us denote by $Q_n \subseteq T_nN$ the Riemannian orthogonal complement of $\mathfrak{g}_N(n)$ in $T_nN_0$. Thus, $Q_n$ is a complex subspace of $T_nN_0$. The Riemannian orthogonal complement of $T_nN_0 \subseteq T_nN$ is $\tilde{J}_n(\mathfrak{g}_N(n))$. Therefore, we have orthogonal direct sum decompositions
\[ T_n N = T_nN_0 \oplus \tilde{J}_n\left(\mathfrak{g}_N(n)\right),\quad T_nN_0 = Q_n \oplus \mathfrak{g}_N(n)\,.\]
If $n \in N_0$ and $w \in T_nN$, we shall decompose $w$ as $w = w_v + w_h + w_t$, where $w_v \in \mathfrak{g}_N(n)$, $w_h \in Q_n$, $w_t \in \tilde{J}_n(\mathfrak{g}_N(n))$. The labels stand for vertical, horizontal and transverse. The following theorem is proved in \cite{pao-jsg0}.

\begin{thm}[Theorem $1.1$ of \cite{pao-jsg0}] \label{thm:pao}
	Suppose that ${0}\in \mathfrak{g}^{\vee}$ is a regular value of $\tilde{\Phi}$, the action of $G$ on $N_0$ is free and $y\in Y$ such that $\tilde{\Phi}(\tilde{\pi}(y))=0$. Let us choose a system of Heisenberg local coordinates centred at $y$. For every $w,\,v\in T_{\tilde{\pi}(y)}N$, the following asymptotic expansion holds as $k\rightarrow +\infty$
	\begin{align*}
	\tilde{\Pi}_{0,k}\left( y+\frac{w}{\sqrt{k}},y+\frac{v}{\sqrt{k}}\right) \sim & \left(\frac{k}{\pi}\right)^{\mathrm{n}-\mathrm{g}/2}\,\frac{2^{\mathrm{g}/2}}{V_{\mathrm{eff}}(y)}\,e^{Q(w_v+w_t, v_v+v_t)}\,e^{\psi_2(w_h,\,v_h)}\, \\
	& \cdot \left[1+\sum_{j=1}^{+\infty} k^{-j/2}\, a_{{\nu}\,j}(x;\,v,\,w)\right]
	\end{align*}
	where $V_{\mathrm{eff}} \,:\, (\tilde{\Phi}\circ \tilde{\pi})^{-1}({0})\rightarrow \mathbb{R}$ is the effective potential of the action, the $a_{{\nu}\,j}$'s are polynomials in $v, w$ with coefficients depending on $y$,
	\[Q(w_v+w_t, v_v+v_t)= -\lVert v_t\rVert_{N}^2- \lVert w_t\lVert_{N}^2 +\imath\left[ \tilde{\omega}_n(w_v,w_t)-\tilde{\omega}_n(v_v,v_t)\right]\,
	\]
	and
	\[\psi_2(w_h, v_h)= -\lVert v_h- w_h\lVert_{N}^2 -\imath\, \tilde{\omega}_n(v_h,w_h)\,.
	\]
\end{thm}

Let us denote with $V_{{\nu}}$ the irreducible component corresponding to ${\nu}$. In view of \eqref{eq:gs2}, integrating over $Y$ the asymptotics of the projector $\tilde{\Pi}_{0,k}$ one has the following theorem.
 
\begin{thm}
 \label{cor:dimension estimate}
Let $G$ be a compact connected Lie group of dimension $\mathrm{g}$, with maximal torus $T\subseteq G$ of dimension $\mathrm{t}$. Assume that $\widetilde{\mu}$ is generically free on $X_{{\nu}}$. Then,
\[
\dim H(X)_{(k\,{\nu},\,k)} \sim  \left(\frac{k}{\pi}\right)^{{d}-\mathrm{t}}\cdot \mathrm{dim}\left(V_{{\nu}}\right) \cdot\mathrm{vol}\left(M_{\mathrm{red}}\right)+O\left(k^{{d}-\mathrm{t}-1/2}\right)\,.
\]
\end{thm}

As already mentioned in the discussion preceding Theorem \ref{cor:dimension estimate}, the space $H^0(M,A^{\otimes k})_{k\,{ \nu }}$ of geometric quantization can be identified with $H(X)_{(k\,{\nu},\,k)}$. In fact, exact formulas for these dimensions are provided, for each given $k$, by the principle ``quantization commutes with reduction'', see \cite{m}. Noticed that for orbits situated far away from the origin the Riemann Roch number of $M_{\mathrm{red}}$ coincides asymptotically with its volume as remarked in \cite{guillemin-sternberg hq}. Under the guise of this article, and also in the more general setting of CR manifold, a vast literature should be recalled. See for example \cite{hmm}, \cite{mz} and \cite{ma} for a survey.

In closing, let us clarify how this work is related to prior literature. In this context one can study different types of asymptotics. With the decomposition \eqref{eq: decomposition} in mind, in the present article we have studied when both $\nu$ and $j$ go to infinity at the same rate. In a similar way one can fix $j$ and study the local and global asymptotics when $\nu$ goes to infinity and vice-versa. The former is a trivial case since $\nu$ appears in the decomposition of $H_k(X)$ if it lies in the image of $j\, \Phi$, but the image of the moment map is a compact subset of $\mathfrak{g}$. The opposite case, fixing $\nu$ and letting $j$ go to infinity, was investigated in \cite{pao-jsg0}.

In \cite{gp-asian}, Paoletti and the present author studied the case $G=SU(2)$ by fixing ${\nu} \in \hat{G}\cong \mathbb{Z}$ and studying the local asymptotics of $\Pi_{k\,{\nu}}$ when $k\rightarrow +\infty$; the perspective is closer in spirit to \cite{pao-IJM}: the structure circle action remains in the background and does not play any privileged role in the asymptotics. With the same spirit in \cite{gp} we studied the case $G=U(2)$; notice that some of the present computations resemble the ones of \cite{gp}. In fact, from this point of view, the approach of this paper is the same but with group $SU(2)\times S^1$: we investigate the local asymptotics of $\Pi_{j,\,{\nu}}$ when $(j,\,{\nu})$ drifts to infinity along various rays in $\mathbb{Z}\times \mathbb{Z}$. Under the same guise in \cite{pao-loa} and \cite{pao-IJM} it is studied the case of circle and torus action respectively.

Our analysis builds on microlocal techniques that can be also applied in the almost complex symplectic setting, see \cite{sz}. For the sake of simplicity, we have restricted our discussion to the complex projective setting.

\bigskip

\textbf{Acknowledgments} I gratefully thank the anonymous reviewer of \cite{gp-asian} and Prof. S. Zeldtich, during the conference GEOQUANT $2019$, for suggesting the study of this problem. I would also like to thank Prof. R. Paoletti for many useful expository improvements and for proposing motivating remarks, in particular the idea of using the shifting trick for computing the asymptotic of the dimension.

\textbf{Funding} This paper was written when the author was a Post-doc researcher at Universit\`a degli Studi di Milano-Bicocca (Progetto ID 2018-ATESP-066). Funding for the conference GEOQUANT $2019$ was supported by University of Luxembourg and by Universit\`a degli Studi di Milano-Bicocca (FA $2017$-ATE-$0253$).

\section{Preliminaries}
\label{sctn:preliminaries}

Some of propositions and lemmas proved in the next subsections are based on ideas contained in \cite{gp} and \cite{gp-asian}.   

\subsection{The geometry of $G/T$}
\label{sctn:local coordinates G/T}

The special unitary group is diffeomorphic to the unit sphere $S^3 \subset \mathbb{C}^2$, explicitly 
\begin{equation}
\label{eqn:S2 and S3}
\begin{pmatrix}
\alpha & -\overline{ \beta } \\
\beta & \overline{ \alpha }
\end{pmatrix} \in G\stackrel{\gamma}{ \longrightarrow }
\begin{pmatrix}
\alpha \\
\beta
\end{pmatrix}\in S^3.
\end{equation}
Furthermore, we shall rely on the local coordinates $(\theta, \delta)\in (0,\pi/2) \times (-\pi,\pi)$ for the coset in $G/T$ of the matrix (\ref{eqn:S2 and S3}) with $\alpha=\cos(\theta) \, e^{ \imath \delta }$, $\beta = \sin (\theta)$.

Finally, let us notice that the function
\[f\left(\begin{pmatrix}
\alpha & \bar{\beta} \\
\beta & \bar{\alpha}
\end{pmatrix} \right)=\begin{pmatrix}
2\,\alpha\,\bar{\beta} \\ |\alpha |^2-| \beta|^2
\end{pmatrix}\in \mathbb{C}\times \mathbb{R} \] 
defines a diffeomorphism between $G/T$ and $S^2\subset \mathbb{R}^3$. In particular we shall use local coordinates defined by $f$ in a local neighborhood of $(0,\,0,\,1)$ in $\mathbb{R}^3$.

\subsubsection{The geometric setting}
\label{sec:geometric}

\begin{lem}
	Suppose that the moment map is transverse to the orbit $\mathcal{O}_{{\nu}}$. The locus $M_{{\nu}}$ is a compact connected hyper-surface in $M$.
	\begin{proof}
		By the transversality assumption $M_{{\nu}}$ is a sub-manifold $M$ of co-dimension $1$. The locus $\mathcal{O}_{{\nu}}$ is closed in $\mathfrak{g}$, hence $M_{{\nu}}$ is compact. By~\cite{Le} each level set of $\Phi$ is connected, since $M_{{\nu}}=G\cdot\Phi^{-1}(\nu\,\beta)$ we proved the lemma.
	\end{proof}
\end{lem}

We denote by $\mathbb{R}_{>0}$ the set of positive real numbers. Recall that we are assuming $M_{{\nu}}\neq \emptyset$ and $M_{\text{out}}=M\setminus \overline{M_{\text{in}}}$, we have the following.

\begin{lem} 	$\partial M_{\mathrm{out}}=M_{{\nu}}$.
	\begin{proof}
		By continuity of $\lambda$ we have $\partial M_{\mathrm{out}}\subset M_{\nu}$.  Indeed if $x \in \partial M_{\mathrm{out}}$, then $x\notin  M_{\mathrm{out}}$ so $\lambda(x)\leq \nu$, and there is a sequence $x_j \in  M_{\mathrm{out}}$ converging to $x$, so that by continuity $\lambda(x) = \lim_j \lambda(x_j) \geq \nu$. Thus $x \in M_{\nu}$.
		
		Let us now prove that  $ M_{\nu}\subset \partial M_{\mathrm{out}}$. Let $x \in  M_{\nu}$, so that $\lambda(x) = \nu$. As $\nu$ is a regular value of $\lambda$, there is a direction $\mathbf{v} \in T_xM$ such that $\mathrm{d}_x\lambda(\mathbf{v}) \in \mathbb{R}^*$. Let $\epsilon=\pm1$ be the sign of $\mathrm{d}_x\lambda(\mathbf{v})$, then for $t \rightarrow 0^+$, one has 
		$$\lambda(\exp_x (t\,\epsilon\mathbf{v}))=\nu+t\,\lvert \mathrm{d}_x\lambda(\mathbf{v})\rvert + O(t^2)$$
		thus $\exp_x (\epsilon t\mathbf{v}) \in M_{\mathrm{out}}$ for $t > 0$ small enough. Hence $x \in \overline{M_{\mathrm{out}}}$, but clearly $x \notin M_{\mathrm{out}}$, so we find $x \in \partial M_{\mathrm{out}}$.  
	\end{proof}
\end{lem}

Fix $m\in M_{{\nu}}$ and suppose that the moment map is transverse to the orbit $\mathcal{O}_{{\nu}}$. The tangent space to $M_{{\nu}}$ at $m$ is given by 
\begin{equation}T_mM_{{\nu}}=\mathrm{d}_m\Phi^{-1} \left(T\mathcal{O}_{\Phi(m)}\right)\,.\label{eq:tg} \end{equation}
Let us denote by $\Upsilon$ a vector field on $M$ spanning the normal bundle to $M_{{\nu}}$, which has real dimension one.
\begin{lem} Under the hypothesis of the previous lemma, 
	\[\Upsilon(m)=J_m(\Phi(m)_M(m))\, \]
	spans the normal bundle to $M_{\nu}$ in $M$.
	\begin{proof} For each $\mathbf{v}\in T_mM_{{\nu}}$, in view of~\eqref{eq:tg}, we have 
		\begin{align*}
			g_m(\mathbf{v},\,J_m(\Phi(m)_M(m)))=\omega_m\left(\mathbf{v},\,\Phi(m)_M(m)\right)=-\langle\mathrm{d}_m\Phi(\mathbf{v}),\,\Phi(m)\rangle=0\,.
		\end{align*}
	\end{proof}
\end{lem}

In conclusion of this section, let us prove the following lemma.
\begin{lem} The vector field $\Upsilon$ is outer oriented.	
\end{lem}
\begin{proof}
	Suppose that $m\in M_{{\nu}}$ and consider the path $\gamma:\, (-\epsilon,\,\epsilon)\rightarrow \mathfrak{g}$ such that
	\[\gamma\,:\, \tau \mapsto \Phi\left(m+\tau\,\Upsilon(m) \right) \]
	defined for sufficiently small $\epsilon$; the expression $m+\tau\,\Upsilon(m) \in M$ is meant in an adapted coordinates system on $M$ centred at $m$. Let us fix the following basis in $\mathfrak{g}$, we have 
	\[\mathcal{B}_{\mathfrak{g}}:=\{\xi_0:=\mathrm{Ad}_{h_m}(\beta),\,\xi_1,\,\xi_2 \}\,; \] 
	where $\{\xi_1,\,\xi_2\}$ is an orthonormal basis (with respect to the standard scalar product in $\mathfrak{g}$) of the orthogonal complement of $\langle \xi_0 \rangle$. Hence, we have
	\begin{equation} \label{eq:gamma}
		\gamma(0)= \lambda(m) \,\mathrm{Ad}_{h_m}\left(\beta \right)=\frac{\nu}{2}\,\mathrm{Ad}_{h_m}\left(\beta \right)
	\end{equation}
	and 
	\begin{equation} \label{eq:gammad}
		\dot{\gamma}(0)=\sum_{j=0}^2\omega_m\left(\xi_j,\,\Upsilon(m) \right)\,\xi_j= \frac{\nu}{2}\,\left\lVert(\mathrm{Ad}_{h_m}(\beta))_M(m)\right\rVert^2 \xi_0\,.
	\end{equation}
	
	Then, by equations~\eqref{eq:gamma} and~\eqref{eq:gammad}, we can easily conclude the following inequality $\lambda(m+\tau\,\Upsilon(m))> \nu/2$.
\end{proof}

\subsubsection{Szeg\H{o} kernels}
\label{sctn:szego fio}

Let $\Pi\,:\,L^2(X)\rightarrow H(X)$ be the Szeg\H{o} projector,  $\Pi(\cdot,\cdot)$ its kernel. By~\cite{bs}, $\Pi$ is a Fourier integral operator with complex phase:
\begin{equation}  \label{eqn:szegokernel}
	\Pi(x,y)=\int_0^{+\infty}e^{iu\psi(x,y)}\,s(x,y,u)\,\mathrm{d}u,
\end{equation}
where the imaginary part of the phase satisfies $\Im (\psi)\ge 0$ and 
$$
s(x,y,u)\sim \sum_{j\ge 0}u^{d-j}\,s_j(x,y).
$$
We shall also making use on the description of the phase $\psi$ in Heisenberg local coordinates (see \S 3 of~\cite{sz}).

\subsubsection{Equivariant Szeg\H{o} kernels and Weyl formulae}
\label{sctn:weylint&char} 

Let us denote by $r_{\theta}\,:\,X\rightarrow X$ the standard circle action on $X$. By composing $\Pi$ with the equivariant projector associated to ${\mu}=(\mu>0)$ (see the discussion in~\cite{guillemin-sternberg hq}) and pick the $j$-th Fourier component ($j\in \mathbb{Z}$), we have 
\begin{equation}
	\label{eqn:equiv_projector}
	\Pi_{j,\mu}
	\left(x,y\right)
	=
	\mu \cdot \int_G \, \mathrm{d} V_G (g) \left[\overline{\chi_{{\mu}} (g)} 
	\,\Pi_j\left( \widetilde{\mu}_{g^{-1}}( x ),
	y \right) \right].
\end{equation}
where we have set
\begin{equation*}
	\Pi_{j}
	\left( \widetilde{\mu}_{g^{-1}}( x ),y \right)
	=
	\frac{1}{2\,\pi} \cdot \int_{-\pi}^{\pi} \mathrm{d}\theta  \left[e^{-\imath\,j\,\theta}\,\Pi\left( \widetilde{\mu}_{g^{-1}}( r_{\theta}(x) ),
	y \right) \right].
\end{equation*}
We can re-manage~\eqref{eqn:equiv_projector} as follows. 
Let us define $F_j:T\rightarrow \mathcal{D}'(X\times X)$ by setting
\begin{equation}
	\label{eqn:defn of F(t)}
	F_j(t;x,y):=\int_{G/T} \, \mathrm{d} V_{G/T} (g\,T) \,
	\left[
	\Pi_j\left( \widetilde{\mu}_{g\,t^{-1}\,g^{-1}}( x ),
	y \right) \right]\quad (t\in T)\nonumber.
\end{equation}

We have $F_j(t;x,y)=F_j\left(t^{-1};x,y\right)$, since $t$ and $t^{-1}$ are conjugate in $G$. Let $t_1$ and $t_2=t_1^{-1}$ denote the diagonal entries of $t\in T$. Then by the Weyl integration and character formulae, see~\cite{var}, we have
\begin{align*}
	\Pi_{j,{\mu}}
	\left(x,y \right)
	&=
	\frac{\mu}{2} \cdot \int_T\,\, \mathrm{d} V_{T} (t) \,
	\left(t_1^{-\mu}-t_1^{\mu}\right)
	\left(t_1-t_1^{-1}\right)\,F_j(t;x,y)\nonumber\\
	&= I_+(\mu;x,y)-I_-(\mu;x,y), \nonumber
\end{align*}
where we have set
\begin{align}
	\label{eqn:defn of I+}
	I_{\pm}(\mu;x,y)
	:=\frac{\mu}{2} \cdot \int_T\,\, \mathrm{d} V_{T} (t) \,\left[t_1^{\mp \mu} \cdot\left(t_1-t_1^{-1}\right)\cdot
	F_j(t;x,y)\right].
\end{align}
In~\eqref{eqn:defn of I+}, the change of variable $t\mapsto t^{-1}$ shows that 
$I_-(\mu;x,y)=-I_+(\mu;x,y)$.
Hence,
\begin{align*}
	\Pi_{j,{\mu}}
	\left(x,y \right)
	&=2\, I_+(\mu;x,y)\\
	&=\mu \cdot \int_T\,\, \mathrm{d} V_{T} (t) \,\left[t_1^{- \mu} \cdot\left(t_1-t_1^{-1}\right)\cdot
	F_j(t;x,y)\right].
\end{align*}

\subsection{Reduction to compactly supported integrals}
\label{sctn:compact reduction}

Let us consider $x,y \in X$. By the discussion \S\ref{sctn:weylint&char}, we can write explicitly
\begin{equation*}
\Pi_{(k {\nu},k) }
\left(x , y \right)
 = \frac{1}{2\,\pi}\,\int_{-\pi}^{\pi} \mathrm{d}{\theta} \left[e^{-\imath\,k\,\theta}\,\Pi_{k\,{\nu}}\left(r_{\theta}(x),\,y\right) \right] \,,  
\end{equation*}
where we have set
\begin{equation} \label{eqn:knu}
\Pi_{k\,{\nu}}\left(r_{\theta}(x),\,y\right)= k\nu \, \int_G \mathrm{d} V_G (g)\, \left[\overline{\chi_{k{\nu}} (g)} \,\Pi\left( \widetilde{\mu}_{g^{-1}}( r_{\theta}(x) ),
y \right) \right]\,. 
\end{equation}

Fix $\theta\in (-\pi,\,\pi)$, and consider the oscillatory integral \eqref{eqn:knu}. For some suitably small $\delta>0$, let us define 
\begin{eqnarray}
 \nonumber
G_{<\delta}(x,y,\theta)&:=&\big\{g\in G\,:\,\mathrm{dist}_X\left(\widetilde{\mu}_{g^{-1}}(r_{\theta}(x)),y\right)<\delta\big\}, \\
G_{>\delta}(x,y,\theta)&:=&\big\{g\in G\,:\,\mathrm{dist}_X\left(\widetilde{\mu}_{g^{-1}}(r_{\theta}(x)),y\right)>\delta\big\}.
\nonumber
\end{eqnarray}

Let us consider a smooth partition of unity $\{\varrho, 1-\varrho\}$ of $G$ subordinate to the cover $\{G_{<2\,\delta}(x,y,\theta),\,G_{>\delta}(x,y,\theta)\}$. One can see that $\varrho=\varrho_{x,y,\theta}$ may be chosen to depend smoothly on $(x,y,\theta)\in X\times X \times (-\pi,\pi)$; we shall omit the dependence on $(x,y,\theta)$.

When $\varrho(g)\neq 1$, we have 
$
\mathrm{dist}_X\left( \widetilde{\mu}_{g^{-1}}( r_{\theta}\left(x\right) ),
y\right) \ge \delta>0$.
Because $\Pi$ is $C^{\infty}$ away from the diagonal, the function 
\begin{equation} \label{eq:1-rho}
g\mapsto \big( 1-\varrho(g)\big)\cdot \Pi\left( \widetilde{\mu}_{g^{-1}}( r_{\theta}\left(x\right) ),
y \right)
\end{equation}
is smooth on $G$. Therefore, inserting \eqref{eq:1-rho} in \eqref{eqn:knu} and arguing as in Section \S \ref{sctn:weylint&char}, we obtain the $k$-th Fourier transform of a smooth function. Thus we have the following proposition.

\begin{prop}
 \label{prop:localization near G_x} If the integrand of 
(\ref{eqn:knu}) is multiplied by $\varrho(g)$, only a rapidly decreasing contribution to the asymptotic is lost.
\end{prop}

On the support of the bump function $\varrho$, the couple $\left( \widetilde{\mu}_{g^{-1}}(r_{\theta}\left(x\right)), y\right)$ lies in a small neighborhood of the diagonal; since any smoothing term will contribute negligibly to the asymptotics, we may replace $\Pi$ by its representation as
an Fourier integral operator (see Section \S \ref{sctn:szego fio}). 
Thus, we insert (\ref{eqn:szegokernel}) in (\ref{eqn:knu})
(with the factor $\varrho (g)$ included), and apply the rescaling $u\mapsto k\,u$ we obtain
\begin{eqnarray}
\label{eqn:szego rescaled proj 1}
\Pi_{k {\nu} }
\big (r_{\theta}(x) , y \big)
& \sim &
k^2\nu \, \int_G \, \mathrm{d} V_G (g) \,\int_0^{+\infty}\,\mathrm{d}u\\
&&\left[\varrho(g)\cdot\overline{\chi_{k{\nu}} (g)} \,
e^{
\imath\,k\,u\,\psi\left( \widetilde{\mu}_{g^{-1}}( r_{\theta}\left(x\right) ),
y \right)
} \cdot s \left( \widetilde{\mu}_{g^{-1}}(r_{\theta}\left(x\right)),
y, k\,u \right) \right].\nonumber
\end{eqnarray}

Integration in (\ref{eqn:szego rescaled proj 1}) can be reduced to a suitable compact domain without altering the asymptotics. This fact was already proved in Section \S 3.5 in \cite{gp-asian}, see especially Proposition 3.2, that we rewrite here for the convenience of the reader.

\begin{prop}
 \label{prop:compact support u}
Let $D\gg 0$ and let $\rho $ be a bump function with compact support such that $\ge 0$, supported in $\big(1/D, D\big)$,
and $\equiv 1$ on $(2/D, D/2)$.
Then only a rapidly decreasing contribution to the asymptotics is lost, 
if the integrand on the last line of (\ref{eqn:szego rescaled proj 1}) is
multiplied by $\rho(u)$. 
\end{prop}

\section{Proof of Theorem \ref{thm:rapid decrease}}
\label{sctn:proof rapid decrease}

\begin{proof}
[Proof of Theorem \ref{thm:rapid decrease}]

By the discussion in Sections \S\ref{sctn:compact reduction} and \S\ref{sctn:weylint&char}, we have the following oscillatory integral
\begin{align} \label{eqn:szego rescaled proj 1 for rapid decrease}
&	\Pi_{(k {\nu},\,k) }
	(x , y ) \sim 
\frac{k^2\nu }{(2\pi)^2}\, \int_{1/D}^{D}\,\mathrm{d}u \,\int_{-\pi/2}^{3\pi/2} \,\mathrm{d}\vartheta \,\int_{G/T} \, \mathrm{d}V_{G/T} (gT)\,\int_{-\pi}^{\pi} \,\mathrm{d}\theta  
\\ 
&\qquad\biggl[e^{\imath\, k\, \Psi_{x,y}(u,\vartheta,gT,\theta)}\cdot \rho(u) \cdot \left( e^{\imath\vartheta}-e^{-\imath\,\vartheta}\right) \cdot
s\left( \widetilde{\mu}_{g e^{- \vartheta \beta} g^{-1}}( r_{\theta}(x) ), y, k\, u \right)
\biggr] \,,\nonumber
\end{align}
where we have set
\begin{equation} \label{eq: phasedec}
\Psi_{x,y}(u,\vartheta,gT,\theta)= u\, \psi \left( \widetilde{\mu}_{g e^{-\imath \vartheta \beta} g^{-1}}( r_{\theta}(x) ),y \right) -\nu\, \vartheta-\theta\,. \end{equation}

\begin{prop} \label{prop:decayorbit} Uniformly for 
	\[\mathrm{dist}_X\left(x, (G\times S^1)\cdot y\right)\ge C\, k^{\epsilon -1/2}\]
	we have $\Pi_{(k {\nu},\,k) }\big (x , y ) =O(k^{-\infty})$ as $k\rightarrow +\infty$.
\end{prop}
\begin{proof}[Proof of Proposition \ref{prop:decayorbit}]
We have
\begin{align*}
\lvert \partial_u\Psi_{x,y}(u,\vartheta,gT,\theta) \rvert &=\left| \psi\left(\widetilde{\mu}_{g^{-1}}(r_{\theta}(x)),y\right)\right|\ge \Im \left(\psi\left(\widetilde{\mu}_{g^{-1}}(r_{\theta}(x)),y\right)\right) \\
&\ge D\,\mathrm{dist}_X\left(\widetilde{\mu}_{g^{-1}}(r_{\theta}(x)),y\right) \ge C^2\,D\,k^{2\,\epsilon-1}
\end{align*}
for some constant $D>0$ depending only on $X$ (see Corollary 1.3 of \cite{bs}). Iteratively integrating by parts in $\mathrm{d}u$, we prove Proposition \ref{prop:decayorbit}.
\end{proof}

Let $\varrho\in \mathcal{C}_0^{\infty}([0,\,+\infty))$ be $\equiv 1$ on $[0, 1)$ and $\equiv 0$ on $\mathbb{R}\setminus(0, 2)$. Let us choose $\epsilon >0$ and multiply the integrand of \eqref{eqn:szego rescaled proj 1 for rapid decrease} by
\[\varrho\left(k^{1/2-\epsilon}\,\mathrm{dist}_X\left(\widetilde{\mu}_{g^{-1}}(r_{\theta}(x)),y\right) \right)+\left[1-\varrho\left(k^{1/2-\epsilon}\,\mathrm{dist}_X\left(\widetilde{\mu}_{g^{-1}}(r_{\theta}(x)),y\right) \right) \right]=1\,. \]
Thus, we obtain the following splitting 
\begin{equation} \label{eq:split}\Pi_{(k {\nu},\,k) }
\big (x , y )=\Pi_{(k {\nu},\,k) }(x , y )^{(1)}+\Pi_{(k {\nu},\,k) }
\big (x , y )^{(2)}\,, \end{equation}
where $\Pi_{(k {\nu},\,k) }(x , y )^{(1)}$ (resp. $\Pi_{(k {\nu},\,k) }(x , y )^{(2)}$) is multiplied by the first summand (resp. the second summand) in equation \eqref{eq:split}. Since on the support of $\Pi_{(k {\nu},\,k) }(x , y )^{(2)}$ we have $\mathrm{dist}_X((G\times S^1)\cdot x,y)\ge k^{\epsilon-1/2}$, arguing as in Proposition \ref{prop:decayorbit} we conclude  $\Pi_{(k {\nu},\,k) }(x , y )^{(2)}=O(k^{-\infty})$.

\begin{rem} \label{rem:sum} Suppose that $(g,e^{\imath\,\theta})\in G\times S^1$ fixes $x\in X$ with $\pi(x)=m$, equivalently $r_{\theta}(x)=\tilde{\mu}_{g}(x)$. Then $g\in G_m$, the stabilizer of $m$ in $M$, and $\theta$ is unique determined by $g$. Suppose that $x=y$ in \eqref{eq:split} and consider $\Pi_{(k {\nu},\,k) }(x ,\, x )^{(1)}$. Since $G_m$ is discrete (see Lemma $3.1$ in \cite{gp-asian}), then we can write $\varrho$ has a sum over $j=1,\dots |G_m|$, as follows
	\[\sum_{j=1}^{|G_m|} \varrho\left(k^{1/2-\epsilon}\,\mathrm{dist}_G\left(g,g_j\right) \right)\cdot \varrho\left(k^{1/2-\epsilon}\,|\theta-\theta_j| \right)\,. \]
	
\end{rem}

In order to complete the proof of Theorem \ref{thm:rapid decrease}, we need to establish the following proposition.
 
\begin{prop} \label{prop:decrXnuj} Let us suppose that 
	\begin{equation} \label{eq:disXnuj}
	\mathrm{dist}_X(x,\,X_{\mathrm{in}}) \geq C\,k^{\epsilon-1/2}\,. 
	\end{equation}
	Then, we have $\Pi_{(k{\nu}, k)}(x,x)=O(k^{-\infty})$.
\end{prop}
	\begin{proof}[Proof of Proposition \ref{prop:decrXnuj}]
	Instead of \eqref{eqn:szego rescaled proj 1 for rapid decrease}, by Remark \ref{rem:sum}, we have
	\begin{eqnarray} \label{eq:pixjnuscal}
	\lefteqn{
		\Pi_{(k {\nu},\,k) }
		\big (x , x )
	} \\ \nonumber
	& \sim & 
	\frac{k^2\nu }{(2\pi)^2}\,\sum_{j=1}^{|G_m|}\, \int_{1/D}^{D}\,\mathrm{d}u \,\int_{-\pi/2}^{3\pi/2} \,\mathrm{d}\vartheta \,\int_{G/T} \, \mathrm{d}V_{G/T} (gT)\,\int_{-\pi}^{\pi} \,\mathrm{d}\theta  
	\\ \nonumber
	&&\biggl[e^{\imath\, k\, \Psi_{x,x}(u,\vartheta,gT,\theta)}\cdot \rho(u) \cdot \varrho\left(k^{1/2-\epsilon}\,\mathrm{dist}_G\left(g,g_j\right) \right)\cdot \varrho\left(k^{1/2-\epsilon}\,|\theta-\theta_j| \right) \\ 
	&& \quad\cdot \left( e^{\imath\vartheta}-e^{-\imath\,\vartheta}\right) \cdot
	s\left( \widetilde{\mu}_{g e^{- \vartheta \beta} g^{-1}}( r_{\theta}(x) ).
	x, k\, u \right)
	\biggr] \,,\nonumber
	\end{eqnarray}
	
	Let us consider the addendum in \eqref{eq:pixjnuscal} with $g_j=\mathrm{Id}$ and $\theta=0$, (the other terms can be treated similarly). Let us pose $x(\vartheta, gT,\theta) := \widetilde{\mu}_{g e^{-\imath \vartheta \beta} g^{-1}}( r_{\theta}(x))$. On the support of the integrand of \eqref{eq:pixjnuscal} we have 
	\begin{align} \label{eq:dpsi}
	\mathrm{d}_{x(\vartheta,gT,\theta)}\psi=\left(\alpha_{x(\vartheta,gT,\theta)},\,-\alpha_x \right) +O\left(k^{\epsilon-1/2}\right)\,.
	\end{align}
	On the other hand, with $m_x := \pi(x)$, we obtain:
	\begin{align} \label{eq:xtheta}
&	\frac{\mathrm{d}}{\mathrm{d}\tau}
	x\left( \vartheta +\tau, gT,\theta\right)_{\big|_{\tau=0}}
	= -\mathrm{Ad}_{g} (\beta)_X \left(x(\vartheta, gT,\theta)\right) \\
	&\qquad = -(\mathrm{Ad}_{g} (\beta))_M^{\sharp}\left(m_{x}(\vartheta, gT,\theta)\right)  + \langle \mathrm{Ad}_{g^{-1}}\left(\Phi(m_{x}(\vartheta, gT,\theta)\right)),\, \beta \rangle\,\partial_{\theta}\, .\nonumber
	\end{align}
	
	Let $\mathrm{d}^{\vartheta}$ denote the differential with respect to the variable $\vartheta$.  By composing \eqref{eq:xtheta} with \eqref{eq:dpsi}, one easily obtains the following expression for the differential of the phase \eqref{eq: phasedec},
	\begin{equation} \label{eq:varthetadiff}
	\mathrm{d}^{\vartheta}_{(u,\vartheta,gT,\theta)}\Psi_{x,\,x}= u\,\langle \mathrm{Ad}_{g^{-1}}\left(\Phi(m_{x}(\vartheta, gT,\theta)\right)),\, \beta \rangle-\nu +O(k^{\epsilon-{1/2}}) \,.
	\end{equation}
	
	In order to write more explicitly \eqref{eq:varthetadiff}, we can introduce coordinates on $G/T$ as in Section \S \ref{sctn:local coordinates G/T} and write 
	$$\langle \mathrm{Ad}_{g^{-1}}(\Phi(m_{x})),\, \beta \rangle= \lambda(m_x)\,\cos(2\,\theta_g)\,.$$ 
	Finally, by the discussion in \S2.1.3 of \cite{pao-tf}, the condition \eqref{eq:disXnuj} implies that there is a constant $b_{\nu} > 0$ such that, for every $u \in [1/(2D), 2D]$, we have
	\begin{equation} \label{eq:ineq}  \lambda(m_x)\,\cos(2\,\theta_g)-\nu  \geq b_{\nu}\,C\,k^{\epsilon-1/2}\,. \end{equation}
	Thus the norm of \eqref{eq:varthetadiff} can be estimated by inequality \eqref{eq:ineq} and we can prove Lemma \ref{prop:decrXnuj} essentially by iteratively integrating by parts in $\mathrm{d}\vartheta$.
\end{proof}
 	Thus, the statement of Theorem \ref{thm:rapid decrease} holds true when $x = y$. The general case follows	from this and the Schwartz inequality:
 	\[\lvert \Pi_{(k {\nu},\,k) }(x,\,y) \rvert \leq \sqrt{\Pi_{(k {\nu},\,k) }(x,\,x)}\,\sqrt{\Pi_{(k {\nu},\,k) }(y,\,y)}\,; \]
 	in fact, if say \eqref{eq:disXnuj} holds, the first factor is rapidly decreasing and both factors on the right-hand side have at most polynomial growth in $k$ by the following lemma.

\begin{lem} \label{lem:pol}
	There is a constant $C_{\nu} > 0$ such that for any $x\in X$ one has
	\[\lvert \Pi_{(k {\nu},\,k) }(x,\,x) \rvert\leq C \,k^{d} \]
	for $k\gg 0$.
\end{lem} 
\begin{proof}[Proof of Lemma \ref{lem:pol}] Let us recall that
\[H_{(k {\nu},\,k) }(X)=H_k(X)\cap H_{k {\nu} }(X); \]
thus
\begin{equation*}\Pi_{(k {\nu},\,k) }(x,\,x) \leq \Pi_k(x,\,x) \leq 2\,\left(\frac{k}{\pi}\right)^d\,,  \end{equation*}
where the last inequality holds in view of the well-known asymptotic expansion of $\Pi_k(x,\,x)$ from \cite{z}.
\end{proof}

\end{proof}

\section{Proof of Theorem \ref{thm:diagonal}}
\label{sctn:proof rescaled}

\begin{proof}[Proof of Theorem \ref{thm:diagonal}]
Define $\Upsilon$ as in Section \ref{sec:geometric}. For $\tau$ sufficiently small, we set 
\[x_{k,\tau}=x+\frac{\tau}{\sqrt{k}}\Upsilon(m)\,. \]
We are going to produce an asymptotic expansion for $\Pi_{(k{\nu},\,k) }(x_{k,\tau} , x_{k,\tau} )$. As a consequence, setting $\tau=0$, we will obtain the desired asymptotic along the diagonal. By the discussion in Section \S\ref{sctn:weylint&char}, let $\beta$ be as in \eqref{eqn:defn beta}, we have
\begin{align*}
&\Pi_{(k {\nu},\,k) }
 (x_{k,\tau} , x_{k,\tau} )
 \sim 
\frac{k^2\nu }{(2\pi)^2}\, \int_{1/D}^{D}\,\mathrm{d}u \,\int_{-\pi/2}^{3\pi/2} \,\mathrm{d}\vartheta \,\int_{G/T} \, \mathrm{d}V_{G/T} (gT)\,\int_{-\pi}^{\pi} \,\mathrm{d}\theta 
\nonumber\\
&\qquad\left[e^{\imath\, k\, \left[u\, \psi \left( \widetilde{\mu}_{g e^{-\imath \vartheta \beta} g^{-1}}( r_{\theta}(x_{k,\tau}) ),
x_{k,\tau} \right) -\nu\, \vartheta-\,\theta\right]}\cdot \left( e^{\imath\vartheta}-e^{-\imath\,\vartheta}\right) \right.\nonumber\\
&\qquad\left.\cdot \rho(u)\cdot \varrho_1 \left(g e^{- \vartheta \beta} g^{-1}\right)\varrho_2\left({\theta}\right) \cdot
s\left( \widetilde{\mu}_{g e^{- \vartheta \beta} g^{-1}}( r_{\theta}(x_{k,\tau}) ),
x_{k,\tau}, k\, u \right)
\right] .\nonumber
\end{align*}
The bump function $\varrho_j$, $j=0,\,1$, is supported in a small neighborhood of the identity respectively in $G$ and $S^{1}$ and $\rho$ is supported in $(1/D,\,D)$.

Let us fix constants $C_1>0$, $\epsilon_1\in (0,1/6)$. 
Iteratively integrating by parts in $\mathrm{d}u$, similarly to the proof 
of Theorem \ref{thm:rapid decrease}, we conclude that  
the locus where $|\vartheta|>C_1 \, k^{\epsilon_1-1/2}$ 
contributes negligibly to the asymptotics of $\Pi_{(k {\nu},\,k) }(x_{k,\tau}  , x_{k,\tau}  )$.  Similarly it is true for the locus where $|\theta|>C_2 \, k^{\epsilon_2-1/2}$, for some $C_2>0$, $\epsilon_2\in (0,1/6)$.
Hence we conclude the following.

\begin{lem}
\label{lem:shrinking support vartheta}
Suppose that $\varrho_1,\,\varrho_2\in \mathcal{C}_c(\mathbb{R})$ are $\ge 0$, supported in $(-2, 2)$, and $\equiv 1$ on $(-1, 1)$. Then 
the asymptotics are unchanged, if the integrand is multiplied 
by $\varrho_1\left(k^{1/2-\epsilon_1}\,\vartheta\right)\cdot \varrho_2\left(k^{1/2-\epsilon_2}\,\theta\right)$. 
\end{lem}

Applying the rescaling $\vartheta \mapsto \vartheta/ \sqrt{k}$ and $\theta \mapsto \theta/ \sqrt{k}$, we recover
\begin{eqnarray}
\label{eqn:Piknu and szego2}
\lefteqn{
\Pi_{(k {\nu},\,k) }
 (x_{k,\tau} , x_{k,\tau} )
} \\
&\sim & \frac{k\,\nu }{(2\pi)^2}\, \int_0^{+\infty}\,\mathrm{d}u \,
\int_{-\infty}^{+\infty} \,\mathrm{d}\vartheta \,
\int_{G/T} \, \mathrm{d}V_{G/T} (gT)\,\int_{-\infty}^{+\infty}\,\mathrm{d}\theta\,
\left[e^{\imath\, k\, \Psi_{x_{k,\tau}} (u, \theta ,\vartheta, g\,T)} \right.
\nonumber\\
&& \cdot \rho  ( u )\cdot \varrho_1\left(k^{-\epsilon_1}\,\vartheta\right)  \cdot \varrho_2\left(k^{-\epsilon_2}\,\theta\right)\cdot
\left( e^{\imath\vartheta/\sqrt{k}}-e^{-\imath\,\vartheta/\sqrt{k}}\right)\nonumber\\
&&\left.
\cdot 
s\left( \widetilde{\mu}_{g e^{- \vartheta \beta/\sqrt{k}} g^{-1}}( r_{\theta/\sqrt{k}}(x_{k,\tau}) ),
x_{k,\tau}, k\, u \right)
\right] \nonumber,
\end{eqnarray}
where we have set
\begin{equation*}
\Psi_{x_{k,\tau}} (u, \theta ,\vartheta, g\,T)  := 
u\, \psi \left( \widetilde{\mu}_{g e^{- \vartheta \beta/\sqrt{k}} g^{-1}}( r_{\theta/\sqrt{k}}(x_{k,\tau}) ),
x_{k,\tau} \right) -\frac{\vartheta}{\sqrt{k}}\, \nu-\frac{\theta}{\sqrt{k}}\, .
\end{equation*}

Let us make the phase $\Psi_{x_{k,\tau}} $ more explicit by making use of Corollary $2.2$ of \cite{pao-IJM}. With $m_x=\pi(x)$, we have
\begin{align*}
\widetilde{\mu}_{g e^{-\imath \vartheta B/\sqrt{k}} g^{-1}}( r_{\theta/\sqrt{k}}(x_{k,\tau}) )
&=  \widetilde{\mu}_{e^{-\vartheta \mathrm{Ad}_g( \beta)/\sqrt{k}} }( r_{\theta/\sqrt{k}}(x_{k,\tau}) ) \\
& = 
x + \left (  \Theta_k (\tau,\theta,\vartheta , g\,T), 
 \,\mathrm{V}_k(\tau,\, \vartheta , gT)
\right), \nonumber
\end{align*}
where (for appropriate $R_2$ and $R_3$ vanishing respectively at second and third order at the origin)
\begin{align*}
\Theta_k ( \tau,\theta,\vartheta , g\,T) 
& : = \frac{1}{\sqrt{k}}\,\theta+
\frac{1}{\sqrt{k}} \, 
\vartheta \cdot
\Big\langle\Phi (m),  \mathrm{Ad}_g( \beta) \Big\rangle 
\\
& + \frac{1}{k}\,\tau\cdot \vartheta\cdot \omega_{ m } \left( (\mathrm{Ad}_g( \beta))_M (m), \Upsilon(m) \right)  \nonumber\\ 
&
+R_3 \left(\frac{1}{\sqrt{k}} \, \vartheta,\frac{1}{\sqrt{k}} \, \theta, \frac{1}{\sqrt{k}}\,\tau \right),
\nonumber
\end{align*}
and 
\begin{align*}
 \label{eqn:defn of Vk}
\mathrm{V}_k(\tau,\vartheta , gT)
 &: =
\frac{\tau}{\sqrt{k}}\,\Upsilon(m)- \frac{1}{\sqrt{k}}\vartheta \,(\mathrm{Ad}_g\left(\beta\right))_M(m_x) \\
&+R_2 \left(\frac{1}{\sqrt{k}} \, \vartheta,\frac{1}{\sqrt{k}} \, \theta ,\frac{1}{\sqrt{k}}\,\tau \right). \nonumber
\end{align*}
We shall use the abridged notation $\Theta_k$ and $V_k$. In view of \S 3 of \cite{sz} (see especially (65)), we obtain the following expression for the phase of \eqref{eqn:Piknu and szego2}:
\begin{eqnarray}
 \label{eqn:Psi_kexpanded}
\Psi_{x_{k,\tau}} (u, \tau ,\vartheta, g\,T) 
& = &   \imath \, u \,  \left[ 1 - e^{ \imath \, \Theta_k}\right]  -\frac{\vartheta}{\sqrt{k}}\, \nu -\frac{\theta}{\sqrt{k}}\,  
-\imath \, u  \, \psi_2 \left( V_k,
\tau\,\Upsilon(m)\right) \nonumber \\
&& +R_3\left(\frac{1}{\sqrt{k}} \vartheta ,\frac{1}{\sqrt{k}} \, \theta,\frac{1}{\sqrt{k}}\,\tau\right).
\end{eqnarray}

Before stating the next lemma let us introduce a one more piece of notation. Let us set
\begin{equation}
\label{eqn:defn di A}
 A (\theta, \vartheta , g\,T)
:= 
\theta+ \vartheta \cdot 
\left\langle\Phi (m_x),  \mathrm{Ad}_g( \beta) \right\rangle .
\end{equation}
By a few computations, we obtain the following.

\begin{lem}
 \label{lem:expansion of Psik}
We have
\begin{align*}
 \Psi_{x_{k,\tau}} (u, \tau,\theta,\vartheta, g\,T)  : =\, &
\frac{1}{ \sqrt{k} } \, \mathcal{G}_{\tau} (u, \theta,\vartheta, g\,T)
+ \frac{ 1 }{ k } \, \mathcal{D}  (u, \tau ,\theta,\vartheta, g\,T)\\
&   + R_3 \left(\frac{1}{\sqrt{k}} \, \vartheta , \frac{1}{\sqrt{k}} \, \theta, \frac{1}{\sqrt{k}}\,\tau \right),
\end{align*}
where 
\begin{align*}
&\mathcal{G}(u,\theta, \vartheta, g\,T)
 \\
 &\quad= \, u \, A(\theta, \vartheta , g\,T) -\vartheta \, \nu -\theta\, 
\\
 &\quad=\,
\vartheta \cdot \left[  \left\langle \Phi( m_x ), \mathrm{ Ad }_g ( \beta ) \right\rangle \cdot u - \nu \right] +\theta\cdot \left(u-1\right),
\end{align*}
and
\begin{align*}
\mathcal{D}_{\tau} (u ,\theta,\vartheta, g\,T)
 =&  \frac{\imath\, u}{2} \cdot \left( \,A(\theta, \vartheta , g\,T)^2 + \vartheta^2\,\lVert (\mathrm{Ad}_g(\beta))_M(m_x)\rVert^2 \right) \\
 &+2\, \tau\cdot \vartheta\cdot \omega_{ m_x } \left( (\mathrm{Ad}_g( \beta))_M (m), \Upsilon(m) \right) \,.
\end{align*}
with $A$ as in equation \eqref{eqn:defn di A}.
\end{lem}

From (\ref{eqn:Piknu and szego2}) and Lemma \ref{lem:expansion of Psik}, we conclude that
\begin{align}
\label{eqn:Piknu and szego3}
&\Pi_{(k {\nu},\,k) }
(x_{k,\tau} , x_{k,\tau} )
\,
\sim  \frac{k\,\nu }{(2\pi)^2}\, \int_0^{+\infty}\,\mathrm{d}u \,
\int_{-\infty}^{+\infty} \,\mathrm{d}\theta\,
\int_{-\infty}^{+\infty} \,\mathrm{d}\vartheta \,\int_{G/T} \, \mathrm{d}V_{G/T} (gT)
\\
&\left[e^{\imath\, \sqrt{k}\, \mathcal{G}(u, \theta, \vartheta, g\,T) } \cdot e^{\imath \, \mathcal{D}_{\tau}  (u,\theta,\,\vartheta, g\,T)}\cdot e^{\imath\,k\cdot R_3\left(\frac{1}{\sqrt{k}} \, \vartheta ,\frac{1}{\sqrt{k}} \, \theta ,  \frac{1}{\sqrt{k}}\,\tau \right)} \cdot \varrho_1\left(k^{-\epsilon_1}\,\vartheta\right) \cdot \varrho_2\left(k^{-\epsilon_2}\,\theta\right) 
\right.
\nonumber \\
& \left.\quad\cdot \rho  ( u )  \cdot \left( e^{\imath\vartheta/\sqrt{k}}-e^{-\imath\,\vartheta/\sqrt{k}}\right)
\cdot s\left( \widetilde{\mu}_{g e^{-\vartheta \beta/\sqrt{k}} g^{-1}}( r_{\theta/\sqrt{k}}(x_{\tau,k }) ),
x_{k,\tau }, k\, u \right)
\right]\,. \nonumber
\end{align}

We can make (\ref{eqn:Piknu and szego3}) yet more explicit. Let $h_m \,T\in G/T$ be as in (\ref{eqn:defn of hm}), and operate the change of variable $g\,T \mapsto h_m \,g\,T$ in $G/T$; then the phase is
\begin{align}
 \label{eqn:Gamma riscritta}
&\mathcal{G}(u, \theta,\vartheta, h_m g\,T) \\
 &\quad = \vartheta \cdot   \left[ 
u\cdot \left \langle \imath\,g^{-1}
\begin{pmatrix}
 \lambda (m_x) & 0 \\
0 & -\lambda (m_x)
\end{pmatrix}
\, g
,  \beta  \right\rangle - \nu  \right]+\theta\cdot\left(u-1 \right)
  \nonumber\\
&\quad =  \nu\cdot\vartheta \cdot \left[ 
u \cdot (\lvert \alpha\rvert^2-\lvert \beta\rvert^2) - 1\right]+\theta\cdot\left(u-1 \right)
, \nonumber
\end{align}
where we use that $\lambda(m_x)= \nu/2$.

The phase \eqref{eqn:Gamma riscritta} has critical points only if $\lvert \alpha\rvert^2-\lvert \beta\rvert^2=1$, which corresponds to a point $\bar{g}T$ in $G/T$. We can identify $G/T$ with $S^2$ via the map $f$ defined in Section \S\ref{sctn:local coordinates G/T}. Hence, the image $f(\bar{g}T)\in S^2\subseteq \mathbb{R}^3$ of the critical point is $(0,\,0,\,1)$ and we can parametrize a neighbourhood of $\bar{g}T$ in $G/T$ by projecting a small neighbourhood of $(0,\,0,\,1)$ in $S^2$ on the plane $x_3=0$. The volume form is now
\[f^*\left(\mathrm{dV}_{G/T}\right)(x_1,\,x_2)=\mathcal{V}(x_1,\,x_2)\,\mathrm{d}x_1\,\mathrm{d}x_2\,. \]

\begin{lem}  Let us denote with $(r,\,\delta)$ the polar coordinates in a neighborhood of the origin in $\mathbb{R}^2$. The function $\mathcal{V}$ is rotationally invariant and thus we shall write
	\[\mathcal{V}(r)=D_{G/T}\cdot \mathcal{S}_{G/T}(r) \]
where $\mathcal{S}_{G/T}(r)=1+\sum_j s_j\,r^{2\,j}$ and $D_{G/T}$ is as in \eqref{eqn:DG/T}.
\end{lem}

Let us set $\kappa(r):=\lvert \alpha\rvert^2-\lvert \beta\rvert^2=\sqrt{1-r^2}$. Thus, the phase can be rewritten in the following form
\begin{align}
\mathcal{G}(u, \theta,\vartheta, h_m g\,T)&= \nu\cdot\vartheta \cdot \left( \kappa(r)\, u  - 1\right) +\theta\cdot\left(u-1 \right) \, ,
\nonumber
\end{align}
where
\begin{align*} 
{\kappa}(r)
 =  \sum_{j\geq0} b_j\, r^{2\,j} \,;
\end{align*}
notice that $b_0=1$ and $b_1=-1$.

We can rewrite the oscillatory integral \eqref{eqn:Piknu and szego3} in the following way:
\begin{eqnarray}
\label{eqn:Piknu and szego4}
\lefteqn{
	\Pi_{(k {\nu},\,k) }
	(x_{k,\tau} , x_{k,\tau} )
} \\
&\sim & \frac{k\,\nu }{4\pi^2}\, \int_0^{+\infty}\,\mathrm{d}u \,
\int_{-\infty}^{+\infty} \,\mathrm{d}\theta \,\int_{-\infty}^{+\infty} \,\mathrm{d}\vartheta \,\int_{0}^{2\,\pi} \, \mathrm{d}\delta\,\int^{+\infty}_{0} \, \mathrm{d}r\,\,r\,\mathcal{V}(r)
\nonumber\\
&&\left[e^{\imath\, \sqrt{k}\,\mathcal{G}(u, \theta,\vartheta, g\,T) } \cdot e^{\imath \, \mathcal{D}_{\tau} (u ,\theta, \vartheta, g\,T)}\cdot  e^{\imath\,k\cdot R_3\left(\frac{1}{\sqrt{k}} \, \theta,\frac{1}{\sqrt{k}} \, \vartheta, \frac{1}{\sqrt{k}}\,\tau\right)}\right. \nonumber\\
&&\quad \cdot \varrho_1\left(k^{-\epsilon_1}\,\vartheta\right) \cdot \varrho_2\left(k^{-\epsilon_2}\,\theta\right) \cdot \rho  ( u )
\nonumber \\
&& \left.\quad  \cdot \left( e^{\imath\vartheta/\sqrt{k}}-e^{-\imath\,\vartheta/\sqrt{k}}\right)
\cdot s\left( \widetilde{\mu}_{g e^{-\vartheta \beta/\sqrt{k}} g^{-1}}(r_{\theta/\sqrt{k}}(x_{k,\tau}) ),
x_{k,\tau}, k\, u \right)
\right]\,. \nonumber
\end{eqnarray}
where (with abuse of notation) $g=h_m\,g(w)$ by the obvious change of variables.

We can now expand the integrand of \eqref{eqn:Piknu and szego4} in decreasing half integer powers of $k$. First, let us remark that
\begin{align*}  e^{\imath\,\vartheta/\sqrt{k}}-e^{-\imath\,\vartheta/\sqrt{k}} 
=  \frac{2\,\imath}{\sqrt{k }}\cdot \vartheta \cdot \left[ 1 + R_2\left(\frac{\vartheta}{ \sqrt{k} } \right) \right]\,.  
\end{align*}
Then, working in Heisenberg local coordinates, Taylor expansion yields an asymptotic expansion 
\begin{equation*}
s\left( \widetilde{\mu}_{g e^{-\imath \vartheta B/\sqrt{k}} g^{-1}}(r_{\theta/\sqrt{k}} (x) ),
x, k u \right) \sim 
\left( \frac{k u}{\pi} \right)^d \cdot \left[ 1 
+ R_1\left(\frac{\vartheta}{ \sqrt{k} } ,\frac{\theta}{ \sqrt{k} },\frac{\tau}{ \sqrt{k} }, \frac{r}{ \sqrt{k} }\right) \right].
\end{equation*}
As a consequence, we have an asymptotic expansion 
\begin{align*}
&e^{\imath\,k\cdot R_3\left(\frac{1}{\sqrt{k}} \,\vartheta, \frac{1}{ \sqrt{k} }\,\theta, \frac{1}{ \sqrt{k} }\,\tau,\frac{1}{\sqrt{k}} \,r  \right)} 
\cdot \left( e^{\imath\,\vartheta/\sqrt{k}}-e^{-\imath\,\vartheta/\sqrt{k}}\right)
\cdot s\left( \widetilde{\mu}_{g e^{- \vartheta \beta/\sqrt{k}} g^{-1}}(r_{\theta/\sqrt{k}} x ),
x, k\, u \right)
\\
&\qquad \sim \, \left( \frac{k \, u}{\pi} \right)^d \cdot \frac{2\,\imath}{\sqrt{ k }}\cdot \vartheta \cdot 
\left[1+\sum_{j\ge 1} k^{-j/2} \, P_j(x,u; \vartheta,\theta,r )\right], \nonumber
\end{align*}
where $P_j(x,u;\vartheta,\theta)$ is a polynomial of degree $\le 3j$ and parity $(-1)^j$.

Therefore, the amplitude in (\ref{eqn:Piknu and szego4}) is given by an asymptotic expansion in descending half-integer powers of $k$, and the asymptotic expansion for the integrand may be integrated terms by term. By a few computations, one sees that the dominant term of the resulting expansion for (\ref{eqn:Piknu and szego4}) is the dominant term of the expansion for the following oscillatory integral:
\begin{align} \label{eq:final}
&\frac{{k}\,\nu }{4\pi^2}\,\left( \frac{k}{\pi} \right)^d\,\int_{\infty}^{+\infty} \mathrm{d}\vartheta \,\int_{0}^{2\,\pi} \, \mathrm{d}\delta\, \int^{+\infty}_{0}\,\mathrm{d}r\, \\
& \qquad\,\left[e^{-\imath\,\sqrt{k}\,\nu\,\vartheta}\cdot \frac{2\imath\,\vartheta\,}{ \sqrt{k}}\cdot\varrho_1\left(k^{-\epsilon_1}\,\vartheta\right)\cdot I_k(\vartheta,r) \,\cdot\mathcal{V}(r)\,r\right]\,; \nonumber
\end{align}
where we have set
\begin{align*}
I_k(\vartheta,r)= &\int_0^{+\infty}\,\mathrm{d}u \,
\int_{-\infty}^{+\infty} \,\mathrm{d}\theta \\
&
\left[  \varrho_2\left(k^{-\epsilon_2}\,\theta\right) 
\cdot \rho  ( u ) \cdot  u^{d}\cdot e^{\imath\, \sqrt{k}\, \Gamma(u, \theta, \vartheta) } \cdot e^{\imath \, \mathcal{D}_{\tau} (u ,\theta, \vartheta, h_mg\,T)}
\right]
\end{align*}
and
\begin{align*}
\Gamma(u, \theta, \vartheta)=\theta\cdot(u-1)+\nu\cdot \vartheta\cdot \kappa(r)\cdot u\,.
\end{align*}
The unique critical point of the phase ${\Gamma}$ is $P=(1,\,-(\nu\cdot\kappa(r)\cdot\vartheta))$ and its Hessian is non-degenerate in $P$. Thus, by the Stationary Phase Lemma, we have an asymptotic expansion for $I_k(\vartheta,r)$ in decreasing half-integer powers of $k$:
\begin{align} \label{eq:Ik}
I_k(\vartheta,r)\sim e^{\imath\,\sqrt{k}\,\nu\,\vartheta\,\kappa(r) }\cdot\frac{2\,\pi\,}{\sqrt{k}}\,e^{B_{\tau}(\vartheta,r,\delta)}+O(k^{-1})\,,
\end{align}
where
\begin{align*}
B_{\tau}(\vartheta,r)=&-\frac{1}{2}\,\lVert \left(\mathrm{Ad}_{h_mg}(\beta)\right)_M(m) \rVert^2\,\vartheta^2 \\ &-2\,\imath\cdot \tau\,\vartheta\,\omega_m(\left(\mathrm{Ad}_{h_mg}(\beta)\right)_M(m),\Upsilon(m))  \,.
\end{align*}

Inserting the equation \eqref{eq:Ik} in the oscillatory integral \eqref{eq:final}, we obtain the following expression for the leading term of \eqref{eq:final}:
\begin{align} \label{eq:last}
&\imath\,\frac{\nu}{\pi}\,\left( \frac{k}{\pi} \right)^d\,\int_{0}^{2\,\pi} \, \mathrm{d}\delta\, \int^{+\infty}_{0}\,\mathrm{d}r\,  \,\int_{-\infty}^{+\infty} \mathrm{d}\vartheta\, \\
&\qquad \left[ e^{\imath\,\sqrt{k}\,\nu\,\vartheta\,(\kappa(w)-1) }\cdot e^{B_{\tau}(\vartheta,r)} \cdot\vartheta\cdot \varrho_1\left(k^{-\epsilon_1}\,\vartheta\right)\cdot\mathcal{V}(r)\,r\right]\,. \nonumber
\end{align}

Finally, in order to compute the integral \eqref{eq:last}, let us introduce one more piece of notation setting
\[\lambda_{\tau}(r,\delta) := \,\lVert \left(\mathrm{Ad}_{h_mg}(\beta)\right)_M(m) \rVert^2  \]
and
\[\xi_{\tau}(r,\delta)= 2\cdot \tau\, \omega_m(\left(\mathrm{Ad}_{h_mg}(\beta)\right)_M(m),\Upsilon(m)) -\,\sqrt{k}\,\nu\,(\kappa(w)-1)\,. \]
Thus, the integral \eqref{eq:last} can be written as
\begin{align} \label{eq: lastre}
\imath\,\frac{\nu\, }{\pi}\,\left( \frac{k}{\pi} \right)^d\,\int_{0}^{2\,\pi} \, \mathrm{d}\delta\, \int^{+\infty}_{0}\,\mathrm{d}r\,  \left[J_k(r)\cdot\mathcal{V}(r)\,r\right]
\end{align}
where $J_k(r,\delta)$ is the following Gaussian integral
\begin{align} \label{eq:Jk}
J_k(r,\delta)&=\int_{-\infty}^{+\infty} \, e^{-\frac{1}{2}\,\lambda_{\tau}(r,\delta)\,\vartheta^2-\imath\,\xi_{\tau}(r,\delta)\,\vartheta }\cdot \vartheta  \,\mathrm{d}\vartheta \\
&=\sqrt{2\,\pi}\,\frac{(-\imath)}{\lambda_{\tau}(r,\delta)^{3/2}}\cdot \xi_{\tau}(r,\delta)\,e^{-\frac{1}{2\,\lambda_{\tau}(r,\delta)}\,\xi_{\tau}(r,\delta)^2}\,.\nonumber
\end{align}
Inserting \eqref{eq:Jk} in \eqref{eq: lastre} we obtain the following expression
\begin{align} \label{eq:finalPi}
\frac{\sqrt{2}\,\nu\,}{\sqrt{\pi}}\,\left( \frac{k}{\pi} \right)^d\,\int_{0}^{2\,\pi}\,\mathrm{d}\delta\, \int^{+\infty}_{0}\,\mathrm{d}r\,\left[\frac{\xi_{\tau}(r,\delta)}{\lambda_{\tau}(r,\delta)^{3/2}}\cdot e^{-\frac{1}{2\,\lambda_{\tau}(r,\delta)}\,\xi_{\tau}(r,\delta)^2} \cdot\mathcal{V}(r)\cdot r\right] \,.
\end{align}

Setting $\tau=0$, we obtain the desired asymptotic along the diagonal. Explicitly, let us pose
\[S(r)=\kappa(r)-1= -r^2+\sum_{j\geq 2} b_j\, r^{2\,j} \]
and expanding $\mathcal{V}$, $\xi_0$ and $\lambda_0$ in powers of $r$, one notices that the leading term is given by
\begin{align*} 
&2\,\sqrt{2\,\pi}\,D_{G/T}\cdot \frac{\nu^2\,\sqrt{k} }{\lVert (\mathrm{Ad}_{h_m}(\beta))_M(m) \rVert^3}\cdot \left( \frac{k}{\pi} \right)^d\, \int^{+\infty}_{0}\,\mathrm{d}r\, \left[ r^3\cdot e^{-\frac{1}{2\,\lVert \beta_M(m)\rVert^2}\,k\,\nu^2\,r^4} \right] \\
&= \sqrt{2}\,D_{G/T}\,\left( \frac{k}{\pi} \right)^{d-1/2}\cdot \frac{1}{\lVert (\mathrm{Ad}_{h_m}\left(\beta\right))_M(m) \rVert}+O(k^{d-3/4})\,.
\end{align*}
\end{proof}

\section{The shifting trick and the asymptotics of the new projector}

The expression \eqref{eq:finalPi} together with Theorem \ref{thm:rapid decrease} (recall that $\Upsilon$ is outer oriented) can be used to give a lower bound for the asymptotics of the dimension, but we will not dwell on these computations. Instead, we will give a more precise result using a different projector as explained in the next subsections. In the following sections $G$ will be a compact connected Lie group of real dimension $\mathrm{g}$.

\subsection{Rapidly decaying contribution}

In order to compute the asymptotics of the dimension of the corresponding isotypes we need a result concerning asymptotics near the locus $N_0=(\tilde{\Phi}\circ \tilde{\pi})^{-1}({0})$ (for the definitions see the discussion preceding Theorem \ref{cor:dimension estimate} in the Introduction). First, let us recall that, by \cite{pao-jsg0} (see especially equation $(4.15)$), for every $y\in Y$, we can write \eqref{eq:Pitilde} in the following form
\begin{equation} \label{eq:Pitilde1}
\tilde{\Pi}_{0,k}(y,\,y)=k\,\int_{-\pi}^{\pi}\mathrm{d}\theta\int_{0}^{+\infty} \mathrm{d}u\,\int_G \mathrm{dV}_G(g)\left[ e^{\imath\,k\,\Psi_{y,y}}\,\mathcal{B}_{y,y}\right]\, 
\end{equation}
where we have set 
\[
\Psi_{y,\,y}=u\,\psi\left(\tilde{\mu}_{g^{-1}}(r_{\theta}(y)),\,y\right)-\theta\,, \]
\begin{equation} \label{eq:Byy}
\mathcal{B}_{y,y}:=\varrho\left(k^{1/2-\epsilon}\,\mathrm{dist}_Y\left(\widetilde{\mu}_{g^{-1}}(r_{\theta}(y)),y\right) \right) \,\rho(u)\,\tilde{s}\left(r_{\theta}\left(\tilde{\mu}_{g^{-1}}(y)\right),y,k\,u \right)\,, \end{equation}
where $\varrho\in \mathcal{C}_0^{\infty}(\mathbb{R})$ be $\equiv 1$ on $[-1, 1]$ and $\equiv 0$ on $\mathbb{R}\setminus(-2, 2)$ ($\epsilon >0$). 

\begin{rem} Notice that in the equation (4.15) of \cite{pao-jsg0} the amplitude $\mathcal{B}_{y,\,y}$ is written as follows. Instead of $\varrho$ there is the product $b_k(g)\cdot \gamma(\theta)$; where $b_k\in \mathcal{C}^{\infty}(G)$ has support satisfying  \[\mathrm{supp}(b_k)\subseteq B_k =: \{g\in G \,:\, \mathrm{dist}_G\left(g,G_n\right) < 2\,k^{-1/3}  \}\,,  \] and $\gamma\in \mathcal{C}^{\infty}_0((-\pi,\,\pi))$ is supported in $(-\epsilon,\,\epsilon)$, where $\epsilon >0$ is very small (but independent of $k$). Under our assumption, $G_n=\{\mathrm{Id}\}\subseteq G$. In any case, arguing as in Proposition \ref{prop:decayorbit}, instead of $b_k$ and $\gamma$ we can write the bump function $\varrho$ in equation \eqref{eq:Byy}.
\end{rem}

The following lemma gives us a more precise result concerning the asymptotics in an open shrinking neighborhood of $N_0=(\tilde{\Phi}\circ \tilde{\pi})^{-1}({0})$.

\begin{lem} \label{lem: rapN0} Let us suppose that 
	\begin{equation} \label{eq: : rapN0}
	\mathrm{dist}_Y(y,Y_0) \geq C\, k^{\epsilon-1/2}  
	\end{equation}
	then $\tilde{\Pi}_{0,k}(y,y)=O(k^{-\infty})$.
\begin{proof}[Proof of Lemma \ref{lem: rapN0}] 
	Let us denote with $\mathrm{dist}_N$ the distance function on $N$; if $n~=~\tilde{\pi}(y)$, then $\mathrm{dist}_Y(y,Y_0)=\mathrm{dist}_N(n,N_0)$. On the other hand, since ${0}$ is a regular value for $\tilde{\Phi}$, we have
	\[\lVert \tilde{\Phi}(n)\rVert \geq C\,k^{\epsilon-1/2}\,; \]
	as an easy consequence, one has the following lemma (see the proof of Proposition \ref{prop:decrXnuj}, especially equations \eqref{eq:dpsi}, \eqref{eq:xtheta} and \eqref{eq:varthetadiff}). 
	\begin{lem} \label{lem: rapdecph} If \eqref{eq: : rapN0} holds, for any $u\in [1/(2D),\,2D]$ and $k \gg 0$,
	\[\lVert \mathrm{d}^{(\xi)}_{(u,g,\theta)}\Psi_{y,y} \rVert \geq C\,k^{\epsilon-1/2} \]
	on the support of \eqref{eq:Pitilde1}. 
	\end{lem}

Consider the oscillatory integral \eqref{eq:Pitilde1}, it is supported in a small neighborhood of the identity $G$. Let $\exp_G \,:\, \mathfrak{g} \rightarrow G$ be the exponential map, and let $U \subseteq \mathfrak{g}$ be a suitably small open neighborhood of the origin ${0} \in \mathfrak{g}$, over which $\exp_G$ restricts to
a diffeomorphism. We may express the integration in $\mathrm{dV}_G$ using the exponential chart. Let us introduce the following differential operator
\[P= \sum_{h=1}^g \frac{\partial_{\xi_h}\Psi_{y,y}}{\lVert \mathrm{d}^{(\xi)}_{(u,g,\theta)}\Psi_{y,y}  \rVert^2 }\,\frac{\partial}{\partial \xi_h}\,, \]
so that 
\begin{equation}\frac{1}{\imath\,k}\,P\left(e^{\imath k\,\Psi_{y,y}} \right)=e^{\imath k\,\Psi_{y,y}}\,. \label{eq:identity} \end{equation}
Thus, inserting \eqref{eq:identity} in \eqref{eq:Pitilde1} and integrating by parts, we obtain 
\begin{align} \label{eq:intbyparts}
&\int_U e^{\imath\, k\,\Psi_{y,y}}\,\mathcal{B}_{y,y}\, \mathrm{d} {\xi} \\
&\qquad =\frac{1}{\imath\,k}\,\sum_{h=1}^g \int_U   \frac{\partial_{\xi_h}\Psi_{y,y}}{\left\lVert \mathrm{d}^{({\xi})}_{(u,g,\theta)}\Psi_{y,y}  \right\rVert^2 }\,\frac{\partial}{\partial \xi_h} \left[e^{\imath k\,\Psi_{y,y}} \right]\, \mathcal{B}_{y,y} \,\mathrm{d} {\xi} \notag \\
&\qquad =\frac{\imath}{k}\,\sum_{h=1}^g \int_U   e^{\imath k\,\Psi_{y,y}} \frac{\partial}{\partial \xi_h} \left[ \mathcal{B}_{y,y}\,\frac{\partial_{\xi_h}\Psi_{y,y}}{\left\lVert \mathrm{d}^{({\xi})}_{(u,g,\theta)}\Psi_{y,y}  \right\rVert^2 } \right]\,\mathrm{d} {\xi}\,. \notag
\end{align}

For any smooth function $f\in \mathcal{C}^{\infty}(U)$, set 
\[P^t(f):= \sum_{h=1}^g \frac{\partial}{\partial \xi_h} \left[ \frac{\partial_{\xi_h}\Psi_{y,y}}{\left\lVert \mathrm{d}^{({\xi})}_{(u,g,\theta)}\Psi_{y,y}  \right\rVert^2 } \,f \right]\,; \]
iterating \eqref{eq:intbyparts} one obtains the following
\begin{equation} \label{eq: iteration}
\int_U e^{\imath\, k\,\Psi_{y,y}}\,\mathcal{B}_{y,y}\, \mathrm{d} {\xi} =\frac{\imath^r}{k^r}\,\int_U e^{\imath\, k\,\Psi_{y,y}}\,(P^t)^r\left(\mathcal{B}_{y,y}\right)\, \mathrm{d} {\xi}\,. \end{equation}

We shall now study the asymptotics of the oscillatory integral appearing in the right hand side of \eqref{eq: iteration}. We need to introduce some notations. Let us denote with  $$\mathcal{D}=\mathrm{dist}_Y\left(\widetilde{\mu}_{g^{-1}}(r_{\theta}(y)),y\right)\,.$$ 
For $\mathbf{\xi} \sim {0}$, we have
\begin{equation}\label{eq:D}\mathcal{D} = F_1( \xi,\,\theta;\, y) + F_2( \xi,\,\theta;\, y)+\dots\,,\end{equation}
where $F_j( \xi,\,\theta;\, y)$ is homogeneous of degree $j$ in ${\xi}$. Let us set
 \[ \mathcal{D}^{(c)}=\frac{\partial^c\mathcal{D}}{\partial\xi_{i_1}\cdots \partial \xi_{i_c}}\,, \]
where $\mathcal{D}^{(c)}$ is not uniquely determined by $c$. By equation \eqref{eq:D} on the support of \eqref{eq:Byy}, one has
\[ \mathcal{D}^{(c)} = O \left(k^{(c-1)\cdot (1/2-\epsilon)}\right)
\] 
as $k\rightarrow +\infty$, where $\varrho\left(k^{1/2-\epsilon}\,
\mathcal{D}\right) \equiv 1$. Finally, for any multi-index $\mathbf{C} = (c_1, \,\dots ,\, c_s )$, notice that \[\mathcal{D}^{(\mathbf{C})}=\mathcal{D}^{(c_1)}\cdots \mathcal{D}^{(c_s)}=O\left(k^{(1/2-\epsilon')\,\sum_j(c_j-1)} \right)\,. \]

The proof of Lemma \ref{lem: rapN0} follows by Lemma \ref{eq:lemder} below, whose proof we shall omit for the sake of brevity; it is proved by induction in a similar way as Lemma $5.8$ in \cite{gp}.

\begin{lem} \label{eq:lemder}
	For any $r\in \mathbb{N}$, $(P^t)^r\left(\mathcal{B}_{x,x}\right)$ is a linear combination of summand of the form
	\begin{align} \label{eq:summand}
	\varrho^{(b)}\left(k^{1/2-\epsilon'}\,\mathcal{D}_k({\xi}) \right) \, \frac{\mathcal{P}_{a_1}\left(\Psi_{y,y},\,\partial \Psi_{y,y}\right) }{\left\lVert \mathrm{d}^{({\xi})}_{(u,g,\theta)}\Psi_{y,y}  \right\rVert^{2\cdot a_2}}\,k^{b\,(1/2-\epsilon')}\,\mathcal{D}^{(\mathbf{C})}\,,
	\end{align}
	times omitted factors bounded in $k$, where:
	\begin{itemize}
		\item $P_{a_1}$ denotes a generic differential polynomial in $\Psi_{y,y}$, homogeneous of degree $a_1$ in the first derivatives $\partial \Psi_{y,y}$;
		\item if $a:= 2\,a_2-a_1$, then $a,\,b,\,\mathbf{C}$ are subject to the bound 
		\[a+b+\sum_{j=1}^r(c_j-1)\leq 2\,r \]
		(the sum is over $c_j>0$);
		\item $\mathbf{C}$ is not zero if and only if $b>0$.
	\end{itemize}
Here $\varrho^{(l)}$ is the $l$th derivative of the one-variable real function $\varrho$. 
\end{lem}
As $0 < \epsilon' < \epsilon$, the general summand \eqref{eq:summand} is
\[O\left( k^{a(1/2-\epsilon)+[b+\sum_j (c_j-1)](1/2-\epsilon')}\right) = O\left( k^{[a+b+\sum_j (c_j-1)](1/2-\epsilon')}\right)
= O \left(k^{r (1-2\epsilon')}
\right).\]
Making use of \eqref{eq: iteration}, the proof of Lemma \ref{lem: rapN0} is thus complete.
\end{proof}

\end{lem}

In the next section we prove Theorem \ref{cor:dimension estimate} making use of the asymptotics expansion appearing in Theorem \ref{thm:pao}. 

\subsection{Proof of Theorem \ref{cor:dimension estimate}}

\begin{proof}[Proof of Theorem \ref{cor:dimension estimate}]
Let us notice that $\tilde{\Phi}^{-1}({0})$ is the locus 
\[N_0:=\{(m,\,f)\in M\times \mathcal{O}_{{\nu}}\,:\, \Phi(m)=f  \}\subseteq N\,.\]
Under the assumptions of Theorem \ref{thm:pao}, $N_0$ is a compact sub-manifold of $N$, of (real) co-dimension $\mathrm{g}$. Thus the locus on which the kernel $\tilde{\Pi}_{0,\,k}$ localizes is $Y_0:=\tilde{\pi}^{-1}(N_0)$. Let $\mathcal{T} \subset Y$ be a tubular neighborhood of $Y_0$. Thus, we have
\begin{align*}
\dim H_k(Y)_{0} \sim \int_{\mathcal{T}} \tilde{\Pi}_{0,\,k}(y,y)\,\mathrm{dV}_Y (y)\,.
\end{align*} 

For any $y\in Y_0$, we have $T_yY = T_yY_0 \oplus T_y^tY$, where $ T_y^tY$ 
is the transversal subspace to $T_yY_0$ in $T_yY$. The vector space $\tilde{J}_n\left(\mathfrak{g}_N(n)\right)$ is naturally unitarily isomorphic to the transversal space to $Y_0$ (see the Introduction and \cite{pao-jsg0}). 

Furthermore, locally along $Y_{0}$, for some sufficiently small $\delta > 0$ we can parametrize the locus $\mathcal{T}$ by a diffeomorphism $\Gamma$ such that
\[\Gamma\,:\, Y_{0}\times  B_{\mathrm{g}}({0},\,2\,\delta)\rightarrow \mathcal{T},\,\qquad \left(y,\,\mathbf{v} \right) \mapsto y+\mathbf{v} \in Y\,. \]
The latter expression is meant in terms of a collection of smoothly varying systems of Heisenberg local coordinates centered at $y \in Y_{0}$, locally defined along $Y_{0}$.
We shall set $y_{\mathbf{v}} := \Gamma(y,\,\mathbf{v} )$, and write
\[\Gamma^*(\mathrm{dV}_Y )=\mathcal{V}_Y(y,\,\mathbf{v})\,\mathrm{dV}_{Y_{0}}(y)\,\mathrm{d} \mathcal{L}(\mathbf{v}), \]
where $\mathcal{V}_Y\,:\,Y_{0}\times B_{\mathrm{g}}({0},\,2\,\delta) \rightarrow (0,\,+\infty)$ is $\mathcal{C}^{\infty}$.
Hence, we obtain
\begin{align} \label{eq:dim}
\dim H_k(Y)_{0}   \sim \int_{Y_{0}}\mathrm{dV}_{Y_L}(y)\, \int_{ B_g({0},\,2\,\delta)}\,\mathrm{d}\mathcal{L}(\mathbf{v})\,\left[\mathcal{V}_Y(y,\,\mathbf{v})\,\tilde{\Pi}_{0,\,k}(y_{\mathbf{v}},\,y_{\mathbf{v}})\right] \notag
\end{align}

The asymptotics of $\dim H(X)_{(k{\nu},\,k)} $ are unchanged, if the integrand is multiplied by a rescaled cut-off function $\varrho$, where $\varrho$ is identically one sufficiently near the origin in $\mathbb{R}^{\mathrm{g}}$, and vanishes outside a slightly larger neighborhood. Thus, by Lemma \ref{lem: rapN0}, we obtain
\[\dim H_k(Y)_{0}\sim \int_{Y_{0}}\,\mathcal{H}_k(y)\,\mathrm{dV}_{Y_0}(y) \,, \]
where we have set 
\begin{align} \label{eq:Hk}
\mathcal{H}_k(y):= \int_{ B_g({0},\,2\,\delta)}\mathrm{d} \mathcal{L}(\mathbf{v})\,\left[\varrho\left(k^{1/2-\epsilon}\,\tau\right)\,\mathcal{V}_Y\left(y,\,\mathbf{v}\right)\,\tilde{\Pi}_{0,\,k}(y_{\mathbf{v}},\,y_{\mathbf{v}})\right]\,. \end{align}

Rescaling $\mathbf{v} \mapsto \mathbf{v}/\sqrt{k}$ and noticing that the complex dimension of $N$ is $\mathrm{n}=\mathrm{d}+(\mathrm{g}-\mathrm{t})/2$, we can insert in \eqref{eq:Hk} the asymptotic expansion of Theorem~\ref{thm:pao}, we obtain
\begin{align*} 
 &\left(\frac{k}{\pi}\right)^{\mathrm{d}+(\mathrm{g}-\mathrm{t})/2-\mathrm{g}/2}\,\frac{1}{k^{\mathrm{g}/2}\,V_{\mathrm{eff}}(y)}\,\int_{\mathbb{R}^g}\mathrm{d}\mathcal{L}(\mathbf{v}) \,\left[\varrho(k^{-\epsilon}\,\mathbf{v})\,\mathcal{V}_Y\left(y,\,\frac{\mathbf{v}}{\sqrt{k}}\right)\,e^{-2\,\tau^2\,\lVert \mathbf{v}\rVert^2}\,\right]\, \\
 &\qquad \sim \left(\frac{k}{\pi}\right)^{\mathrm{d}+(\mathrm{g}-\mathrm{t})/2-\mathrm{g}}\,\frac{1}{V_{\mathrm{eff}}(y)}+O(k^{\mathrm{d}-(\mathrm{g}+\mathrm{t})/2-1/2})\,. \end{align*}

Hence, we have
\[ \dim H_k(Y)_{0}\sim \left(\frac{k}{\pi}\right)^{\mathrm{d}-(\mathrm{g}+\mathrm{t})/2}\cdot\int_{N_{0}}\frac{1}{V_{\mathrm{eff}}(n)}\, \mathrm{dV}_{N_{0}}(n)+O(k^{\mathrm{d}-(\mathrm{g}+\mathrm{t})/2-1/2})\,. \]
By the definition of $V_{\mathrm{eff}}$, we have $V_{\mathrm{eff}}\,\mathrm{dV}_{M_{\mathrm{red}}}=\mathrm{dV}_{N_{0}}$. Furthermore, making use of Weyl Dimension Formula (see \cite{var}, pg. $32$) notice that $$\dim(V_{k\,{\nu}})= k^{(\mathrm{g}-\mathrm{t})/2}\dim(V_{{\nu}})\,.$$ 
Finally, by \eqref{eq:gs2} we obtain the desired result.
\end{proof}

\end{document}